\documentclass[a4paper]{amsart}
 \usepackage[enableskew]{youngtab}
 \usepackage{amsthm,amsmath,amssymb,xspace,url}
 \newtheorem{thm}{Theorem}[section]
 \newtheorem{lem}[thm]{Lemma}
 
 \newtheorem{prop}[thm]{Proposition}

 \theoremstyle{definition}
 \newtheorem{dfn}[thm]{Definition}

 \theoremstyle{remark}
 \newtheorem*{rmk}{Remark}
 
 \newtheorem*{eg*}{Example}

 \newcommand{\Dfn}[1]{\emph{#1}}                            

 \let\size=\abs

 \newcommand{\shrink}[2][0pt]{\hbox to #1{\hss #2\hss}}

 \DeclareMathOperator{\bigDisjointUnion}{\dot\bigcup}
 \DeclareMathOperator{\disjointUnion}{\mathaccent\cdot\cup}

 \author{Martin Rubey}
 \address{Institut f\"ur Algebra, Zahlentheorie und Diskrete Mathematik, Leibniz
   Universit\"at Hannover, Welfengarten 1, D-30167 Hannover, Germany}
 \email{martin.rubey@math.uni-hannover.de}
 \urladdr{http://www.iazd.uni-hannover.de/~rubey/}

 \keywords{ribbon Schur functions, compositions, Dirichlet series}

 \title{The number of ribbon Schur functions}

 \usepackage[breaklinks=true,hyperref]{hyperref}
 \begin{document}
 \begin{abstract}
   We present a formula for the number of distinct ribbon Schur
   functions of given size and height.
 \end{abstract}
 \maketitle

 \section{Introduction}
 \label{sec:introduction}
 An important basis for the space of homogeneous symmetric functions
 of degree $n$ is the set of \Dfn{Schur functions} $s_\lambda$,
 indexed by partitions $\lambda$ of $n$.  A larger set of homogeneous
 symmetric functions of degree $n$ is the set of \Dfn{skew Schur
   functions} $s_{\lambda/\mu}$, indexed by skew shapes $\lambda/\mu$
 of size $n$, that is pairs of partitions
 $\lambda=(\lambda_1\geq\lambda_2\geq\dots\geq\lambda_k>0)$ of $n+m$
 and $\mu=(\mu_1\geq\mu_2\geq\dots\geq\mu_\ell>0)$ of $m$, such that
 $k$, the number of parts of $\lambda$, is strictly larger than
 $\ell$, the number of parts of $\mu$, and $\mu_i\leq \lambda_i$ for
 $i\leq\ell$.  When $\mu$ is the empty partition,
 $s_{\lambda/\mu}=s_\lambda$.  Since the set of Schur functions is a
 basis, there must be relations between skew Schur functions.
 Equalities between skew functions have been studied by Stephanie van
 Willigenburg, Peter McNamara, Vic Reiner and Kristin
 Shaw~\cite{MR2176598,MR2353252,MR2500893}. So far however, only
 partial results and a conjecture are available.

 The situation is very different for the subset of \Dfn{ribbon Schur
   functions}, that are indexed by \Dfn{ribbons} (also known as rim
 hooks or border strips), i.e., skew shapes that satisfy
 $\lambda_{i+1}=\mu_i+1$ for $i\leq\ell$.  Here are the ribbons of
 size $4$:
 
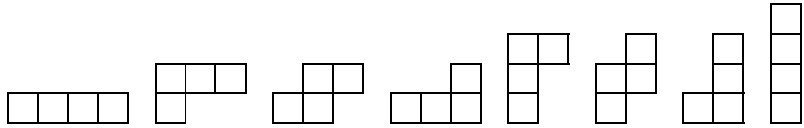
\begin{equation*}
 \begin{tabular}[h]{cccccccc}
   \young(\hfil\hfil\hfil\hfil) & 
   \young(\hfil\hfil\hfil,\hfil) &
   \young(:\hfil\hfil,\hfil\hfil) &
   \young(::\hfil,\hfil\hfil\hfil) &
   \young(\hfil\hfil,\hfil,\hfil) &
   \young(:\hfil,\hfil\hfil,\hfil) &
   \young(:\hfil,:\hfil,\hfil\hfil) &
   \young(\hfil,\hfil,\hfil,\hfil)
 \end{tabular}
 \end{equation*}
 It can be shown that the space of homogeneous symmetric functions of
 degree $n$ is also generated by the set of ribbon Schur functions of
 size $n$.  For these functions, Louis J. Billera, Hugh Thomas, and
 Stephanie van Willigenburg~\cite{MR2233132} give a criterion for
 deciding when they are equal.  In this article we use this criterion
 to count the number of distinct ribbon Schur functions of given size
 and given height, that is, one less than the number of parts of
 $\lambda$.

 Note that ribbons $\lambda/\mu$ of size $n$ and height $m-1$ can be
 identified with compositions $\alpha$ of size $n$ and length $m$ by
 setting $\alpha_i=\lambda_i-\mu_i$ for all $i$.  Two compositions
 $\alpha$ and $\beta$ are called equivalent, denoted
 $\alpha\sim\beta$, if and only if the corresponding ribbon Schur
 functions are equal.

 In the following section we recall a binary operation on compositions
 from~\cite{MR2233132}, that makes the set of compositions into a
 monoid with (almost) unique factorisation.  One of the main theorems
 of~\cite{MR2233132} shows that equivalence of compositions is easily
 determined given their factorisations.

 In Section~\ref{sec:size} we present a relatively appealing formula
 for the number of distinct ribbon Schur functions of given size,
 while in Section~\ref{sec:length} we exhibit a (not nearly as
 beautiful) formula for the number of distinct ribbon Schur functions
 of given size and height.  For more information on symmetric
 functions we refer to Chapter~7 of Enumerative
 Combinatorics~2~\cite{EC2}.

 \section{Composition of Compositions\\ and equality of ribbon Schur
   functions}\label{sec:comp-comp}
 In this section we collect the definitions and results
 from~\cite{MR2233132} that are relevant for our approach.  As
 mentioned before, the basic objects we will be working with are
 compositions:
 \begin{dfn}
   A composition $\alpha$ of a positive integer $m$, denoted
   $\alpha\vDash m$, is a list of positive integers
   $(a_1,a_2,\dots,a_k)$ such that $a_1+a_2+\dots+a_k=m$. We refer to
   each of the $a_i$ as components, and say that $\alpha$ has length
   $l(\alpha) = k$ and size $\size{\alpha} = m$.
 \end{dfn}

 \begin{dfn}
   Let $\alpha=(a_1,a_2,\dots,a_k)\vDash m$ and
   $\beta=(b_1,b_2,\dots,b_\ell)\vDash n$.  Then the
   \Dfn{concatenation} of $\alpha$ and $\beta$ is the composition
   $$\alpha\cdot\beta=(a_1,\dots,a_k,b_1,\dots,b_\ell)\vDash n+m.$$
   Their \Dfn{near concatenation} is
   $$\alpha\odot\beta=(a_1,\dots,a_k+b_1,\dots,b_\ell)\vDash n+m.$$
   Writing 
   $$
   \alpha^{\odot n}=
   \underset{n}{\underbrace{\alpha\odot\alpha\odot\dots\odot\alpha}}
   $$
   we define the \Dfn{composition} of $\alpha$ and $\beta$ as
   $$
   \alpha\circ\beta=\beta^{\odot a_1}\cdot\beta^{\odot
     a_2}\cdots\beta^{\odot a_k}\vDash mn.
   $$
   The composition $\alpha=(a_1,a_2,\dots,a_k)$ is \Dfn{symmetric} if
   it coincides with its \Dfn{reversal}
   $\alpha^*=(a_k,a_{k-1},\dots,a_1)$.
 \end{dfn}

 The following theorem shows that composition of compositions is a
 very well behaved operation indeed:
 \begin{thm}[\cite{MR2233132}, Propositions~3.3, 3.7, 3.8 and 3.9]
   \label{thm:monoid-size-length} The set of compositions together
   with the operation $\circ$ is a monoid, i.e., $\circ$ is
   associative and has neutral element $(1)$.  Furthermore,
   $\size{\alpha\circ\beta}=\size{\alpha}\size{\beta}$ and
   $l(\alpha\circ\beta) = l(\alpha) +
   \size{\alpha}\left(l(\beta)-1\right)$.  Finally,
   $(\alpha\circ\beta)^*=\alpha^*\circ\beta^*$.
 \end{thm}

 Note that composition of compositions is not commutative.  For
 example, $(1,1)\circ (2)=(2)^{\odot 1}\cdot (2)^{\odot 1}=(2,2)$, but
 $(2)\circ(1,1)=(1,1)^{\odot 2}=(1,1)\odot(1,1)=(1,2,1)$.

 \begin{dfn}
   If a composition $\alpha$ is written in the form
   $\alpha_1\circ\alpha_2\circ\dots\circ\alpha_k$ then we call this a
   \Dfn{factorisation} of $\alpha$. A factorisation $\alpha =
   \beta\circ\gamma$ is called trivial if any of the following
   conditions are satisfied:
   \begin{enumerate}
   \item one of $\beta$ and $\gamma$ is the composition $1$,
   \item the compositions $\beta$ and $\gamma$ both have length $1$,
   \item the compositions $\beta$ and $\gamma$ both have all
     components equal to $1$.
   \end{enumerate}

   A factorisation $\alpha_1\circ\alpha_2\circ\dots\circ\alpha_k$ is
   called \Dfn{irreducible} if no $\alpha_i\circ\alpha_{i+1}$ is a
   trivial factorisation, and each $\alpha_i$ admits only trivial
   factorisations.  We call a composition $\alpha$ \Dfn{irreducible},
   if it has not length $1$, not all of its components are equal to
   $1$ and it admits only trivial factorisations.
 \end{dfn}

 \begin{thm}[\cite{MR2233132}, Theorem 3.6]\label{thm:uniqueness}
   The irreducible factorisation of any composition is unique.
 \end{thm}
 It is not surprising that such a theorem is very useful to enumerate
 the underlying objects.  For experimentation it was also of great
 help to have a relatively efficient test for irreducibility, which is
 exhibited in Definition~4.11 and Lemma~4.15
 of~\cite{MR2233132}.\footnote{An implementation can be obtained from
   the author of the present article.}

 Finally, equivalence of compositions and therefore equality of ribbon
 Schur functions is reduced to factorisation by the following theorem.
 Note that it was well known before that reversal of compositions
 yields the same ribbon Schur functions, see for example
 Exercise~7.56 in Enumerative Combinatorics 2~\cite{EC2}, which
 includes also the natural extension to skew Schur functions.
 \begin{thm}[\cite{MR2233132}, Theorem 4.1]\label{thm:equivalence}
   Two compositions $\beta$ and $\gamma$ satisfy $\beta\sim\gamma$ if
   and only if for some $k$, $\beta =
   \beta_1\circ\beta_2\circ\dots\circ\beta_k$ and $\gamma =
   \gamma_1\circ\gamma_2\circ\dots\circ\gamma_k$ where, for each $i$,
   either $\gamma_i = \beta_i$ or $\gamma_i = \beta_i^*$.
 \end{thm}

 \section{The number of ribbon Schur functions of given size}
 \label{sec:size}

 \begin{dfn}
   We order the set of compositions of given length lexicographically.
   Thus, let $\alpha=(a_1,a_2,\dots,a_k)$ and
   $\beta=(b_1,b_2,\dots,b_k)$ be two compositions, then
   $\alpha<\beta$ if and only if $a_s<b_s$ for some $s$, such that
   $a_r=b_r$ for all $r<s$.  $\alpha$ is \Dfn{lexicographic minimal}
   if $\alpha\leq\alpha^*$.

   In view of Theorem~\ref{thm:uniqueness} and
   Theorem~\ref{thm:equivalence}, we call a composition
   \Dfn{normalised}, if all factors in its irreducible factorisation
   are lexicographic minimal.
 \end{dfn}

 Thus, to determine the number of distinct ribbon Schur functions, it
 is sufficient to count normalised compositions.  This is not hard to
 achieve using a suitable combinatorial decomposition.  The validity
 of our decomposition hinges on the following lemma:
 \begin{lem}\label{lem:lexmin}
   Consider a composition $\alpha$ with irreducible factorisation
   $\alpha_1\circ\alpha_2\circ\dots\circ\alpha_k$.  Then $\alpha$ is
   symmetric if and only if all $\alpha_i$ are symmetric for
   $i\in\{1,\dots,k\}$.

   If $\alpha$ is asymmetric, then there is an $\ell\in\{1,\dots,k\}$
   such that $\alpha_\ell$ is asymmetric, and $\alpha_i$ is symmetric
   for all $i>\ell$.  In this situation, $\alpha<\alpha^*$ if and only
   if $\alpha_\ell<\alpha_\ell^*$.
 \end{lem}
 \begin{proof}
   By the last statement of Theorem~\ref{thm:monoid-size-length}, an
   irreducible factorisation of the reversal of $\alpha$ is
   $\alpha^*=\alpha_1^*\circ\alpha_2^*\circ\dots\circ\alpha_k^*$.
   Thus, by Theorem~\ref{thm:uniqueness}, if $\alpha=\alpha^*$, all
   the factors $\alpha_i$ are symmetric.

   Suppose now that $\alpha$ is asymmetric.  Let us first prove that
   for compositions $\beta$, $\gamma$ and $\delta$ with
   $l(\beta)=l(\gamma)$, we have $\beta\circ\delta<\gamma\circ\delta$
   if and only if $\beta<\gamma$.  Namely, if
   $\beta=(b_1,\dots,b_r)<\gamma=(g_1,\dots,g_r)$, then there is an
   index $j$ such that $b_j<g_j$ and $b_i=g_i$ for all $i<j$.  Since
   $\beta\circ\delta=\delta^{\odot b_1}\cdots\delta^{\odot b_r}$ and
   $\gamma\circ\delta=\delta^{\odot g_1}\cdots\delta^{\odot g_r}$, it
   suffices to compare $\delta^{\odot b_j}$ and $\delta^{\odot g_j}$.
   Let $\delta=(d_1,\dots,d_s)$, then the component with index
   $\left(\size{\delta}-1\right)b_j$ of $\delta^{\odot b_j}$, i.e.,
   its last component, equals $d_s$.  However, since $b_j<g_j$, the
   component of $\delta^{\odot g_j}$ with the same index is $d_s+d_1$,
   which is strictly greater than $d_s$.  Hence
   $\beta\circ\delta<\gamma\circ\delta$.  The converse follows by
   symmetry.

   Next, we prove that for compositions $\beta$, $\gamma$, $\delta$
   and $\epsilon$ with $l(\beta)=l(\gamma)$, $l(\delta)=l(\epsilon)$
   $\size{\delta}=\size{\epsilon}$ and $\delta\neq\epsilon$ we have
   $\beta\circ\delta<\gamma\circ\epsilon$ if and only if
   $\delta<\epsilon$.  It suffices to compare the first $r-1$
   components of $\beta\circ\delta$ and $\gamma\circ\epsilon$, which
   are $d_1,d_2,\dots,d_{r-1}$ and $e_1,e_2,\dots,e_{r-1}$
   respectively.  If $\delta=(d_1,\dots,d_r)<\epsilon=(e_1,\dots,
   e_r)$, let $j$ be minimal such that $d_j<e_j$.  Since
   $\size{\delta}=\size{\epsilon}$, the two compositions cannot differ
   only in the last component, so $j\leq r-1$, which implies
   $\beta\circ\delta<\gamma\circ\epsilon$.  Again, the converse
   follows by symmetry.

   To conclude the proof, we write
   $\alpha=\beta\circ\alpha_\ell\circ\gamma$, where $\ell$ is maximal
   such that $\alpha_\ell$ is asymmetric.  (If $\ell=1$ then
   $\beta=(1)$, if $\ell=k$ then $\gamma=(1)$.)  Then
   $\alpha^*=\beta^*\circ\alpha_\ell^*\circ\gamma$.  By the preceding
   two paragraphs, $\alpha<\alpha^*$ if and only if
   $\beta\circ\alpha_\ell<\beta^*\circ\alpha_\ell^*$, which in turn is
   the case if and only if $\alpha_\ell<\alpha_\ell^*$, as desired.
 \end{proof}

 In the following lemma we collect the facts we need about Dirichlet
 generating functions:
 \begin{lem}\label{lem:Dirichlet}
   Let $A$ and $B$ be sets of compositions, let $A\disjointUnion B$ be
   their disjoint union and define $A\circ
   B:=\{\alpha\circ\beta:\alpha\in A, \beta\in B\}$.  For any set of
   compositions $A$, let $A(s)=\sum_{\alpha\in A} \size{\alpha}^{-s}$
   the associated Dirichlet generating function.  Then
   \begin{align*}
     (A\disjointUnion B)(s)&=A(s)+B(s)\quad\text{and}\\
     (A\circ B)(s)&=A(s)B(s).
   \end{align*}
   The latter equality is equivalent to the statement, that the
   coefficient of $n^{-s}$ in $(A\circ B)(s)$ is $a_n*b_n$, where
   $a_n$ and $b_n$ are the coefficients of $n^{-s}$ in $A(s)$ and
   $B(s)$ respectively, and $a_n*b_n$ denotes the Dirichlet
   convolution $\sum_{d|n} a_d b_{n/d}$.
 \end{lem}
 \begin{rmk}
   A full-fledged combinatorial theory of Dirichlet series within the
   theory of combinatorial species was developed by Manuel Maia and
   Miguel M\'endez~\cite{MR2459362}.  Although the proofs below are
   written in the spirit of that theory, they are quite elementary.
 \end{rmk}

 \begin{thm}\label{thm:size}
   The number of normalised compositions of size $n$ is
   $$ 
   2\cdot 2^{n-1} %
   * 2^{\lfloor \frac{n}{2}\rfloor} %
   * \left(2^{n-1}+%
     2^{\lfloor\frac{n}{2}\rfloor}\right)^{-1},
   $$
   where $a_n*b_n$ denotes the Dirichlet convolution, and the
   reciprocal is the inverse with respect to Dirichlet convolution.
 \end{thm}
 \begin{rmk}
   Thus, the numbers of ribbon Schur functions of size $1$ to $33$
   turn out to be:
   \begin{align*}
     & 1, 2, 3, 6, 10, 20, 36, 72, 135, 272, 528, 1052, 2080, 4160,
     8244, 16508, 32896, 65770, \\
     & 131328, 262632, 524744, 1049600, 2098176, 4196200, 8390620,
     16781312, 33558291, \\
     & 67116944, 134225920, 268451240, 536887296, 1073774376,
     2147515424.
   \end{align*}
   (This is sequence \url{http://oeis.org/A120421} in the on-line
   encyclopedia of integer sequences~\cite{OEIS}.\footnote{Be warned
     that the 18\textsuperscript{th} term in the encyclopedia is in
     error, it reads 65768 instead of 65770.})

   It may be interesting to compare the number of ribbon Schur
   functions with the number of lexicographic minimal compositions.
   Since $|\alpha\circ\beta|=|\alpha|\cdot|\beta|$, it is clear that
   the numbers coincide when $n$ is prime.  For $n=9$, there are $136$
   lexicographic minimal compositions, but two of them are equivalent.
   Here are the differences and their positions up to $n=33$:
   \begin{equation*}
   \begin{array}[h]{lrrrrrrrrrrrrrrrrrrrrrrrrrrrr}
     n:&9&12&15&16&18&20&21&24&25&27&28&30&32&33\\
     \text{difference}:&1& 4& 12& 4& 22& 24& 56& 152& 36& 237& 112& 600& 216& 992\\
   \end{array}
   \end{equation*}
 \end{rmk}

 \begin{proof}
   Let $R$ be the set of normalised compositions.  Let $S$ be the set
   of symmetric compositions, $P^\times$ be the set of (normalised)
   asymmetric irreducible compositions and
   \begin{equation}
     \label{eq:R1}
     R^1=P^\times\circ S,
   \end{equation}
   i.e., the set of (normalised) compositions whose first factor in
   the irreducible factorisation is asymmetric, and all remaining
   factors (if any) are symmetric.  We can then decompose the set of
   normalised compositions recursively as
   \begin{equation}
     \label{eq:R}
     R=S\disjointUnion (R\circ R^1),
   \end{equation}
   since a normalised composition is either symmetric, or can be
   written in a unique way as a product of a normalised composition,
   an asymmetric irreducible factor and a symmetric composition.

   The set $R^1$ can be described in terms of the set of all
   compositions $C$ and the set of asymmetric lexicographic minimal
   compositions $L^\times$ by Lemma~\ref{lem:lexmin}.  Namely,
   \begin{equation}
     \label{eq:R1b}
     L^\times=C\circ R^1,
   \end{equation}
   since an asymmetric composition is lexicographic minimal, if and
   only if the last asymmetric factor in its irreducible factorisation
   is lexicographic minimal.

   Finally, we have (again by Lemma~\ref{lem:lexmin}) 
   $$
   2 L^\times=C\setminus S,
   $$
   where $2 L^\times$ is interpreted as the set of asymmetric compositions
   whose last asymmetric factor is either lexicographic minimal or
   lexicographic maximal.

   We can now apply Lemma~\ref{lem:Dirichlet} to obtain the Dirichlet
   generating function for the set of normalised compositions.  We
   have
   \begin{align*}
     L^\times(s)&=1/2\left(C(s)-S(s)\right)\\
     R^1(s)&=L^\times(s)/C(s)
     \intertext{and therefore}
     R(s)&=\frac{S(s)}{1-R^1(s)}\\
         &=\frac{2C(s)S(s)}{2C(s)-\left(C(s)-S(s)\right)}\\
         &=\frac{2C(s)S(s)}{C(s)+S(s)}.
   \end{align*}
   Since $C(s)=\sum_{n\geq1} 2^{n-1} n^{-s}$ and $S(s)=\sum_{n\geq1}
   2^{\lfloor\frac{n}{2}\rfloor} n^{-s}$, the claim follows.
 \end{proof}

 \begin{rmk}
   It is not difficult to obtain more information using the preceding
   theorem and the decompositions in its proof.  In particular, we can
   easily refine the count of normalised compositions by taking into
   account the number of asymmetric irreducible factors.  Denoting the
   number of asymmetric irreducible factors of a composition $\rho$ by
   $\alpha(\rho)$ and defining $R(s, z)=\sum_{\rho\in R} |\rho|^{-s}
   z^{\alpha(s)}$, we find
   \begin{equation*}
     R(s,z)=\frac{S(s)}{1-zR^1(s)}=\frac{2C(s)S(s)}{2C(s)-z\left(C(s)-S(s)\right)}.
   \end{equation*}
 \end{rmk}
 Perhaps more interesting, we can determine the generating function
 for irreducible compositions by size using the following proposition:
 \begin{prop}
   Let $P(s)$ be the Dirichlet generating function for (normalised)
   irreducible compositions, $P^*(s)$ the Dirichlet generating
   function for symmetric irreducible compositions and $R(s)$ the
   Dirichlet generating function for normalised compositions by size.

   Furthermore, let $S(s)=\sum_{n\geq1} 2^{\lfloor\frac{n}{2}\rfloor}
   n^{-s}$ be the Dirichlet generating function of symmetric
   compositions, and $\zeta(s)=\sum_{n\geq1}n^{-s}$ the Riemann zeta
   function.  We then have
   \begin{align}\label{eq:1}
     P(s)   &= 2\zeta^{-1}(s) - 1 - R^{-1}(s)\\
     \label{eq:2}
     P^*(s) &= 2\zeta^{-1}(s) - 1 - S^{-1}(s)
     \intertext{and}\label{eq:3}
     P^\times(s) &= S^{-1}(s)-R^{-1}(s).
   \end{align}
 \end{prop}
 \begin{rmk}
   Thus, the numbers of (normalised) irreducible compositions of size
   $1$ to $33$ are:
   \begin{align*}
     & 0,0,1,2,8,10,34,56,126,234,526,972,2078,4018,8186,16240,32894,65164, \\
     & 131326,261544,524530,1047490,2098174,4191680,8390520,16772994,33557508, \\
     & 67100304,134225918,268416590,536887294,1073708400,2147512258.
   \end{align*}
   Note that, whenever $n$ is prime, there are precisely two
   normalised compositions (or, equivalently, lexicographic minimal
   compositions) that are not irreducible, namely the composition with
   all components equal to $1$ and the composition $(n)$.

   For $n=4$, the irreducible normalised compositions are $(1,3)$ and
   $(1,1,2)$.  For $n=6$, they are $(1,5)$, $(1,1,4)$, $(1,4,1)$,
   $(1,2,3)$, $(2,1,3)$, $(1,1,1,3)$, $(1,1,2,2)$, $(1,1,3,1)$,
   $(2,1,1,2)$, $(1,1,1,1,2)$.
 \end{rmk}
 \begin{proof}
   Let $E$ be the set of compositions with all components equal to
   $1$, and $K$ be the set of compositions with only one component.
   Let $R$ be the set of all normalised compositions, and $R_E$ be the
   set of normalised compositions with no factors in the irreducible
   factorisation having only one component, i.e., all factors being
   irreducible or having all components equal to $1$.  Finally, let
   $P$ be the set of (normalised) irreducible compositions.

   By Theorem~\ref{thm:uniqueness}, $R_E$ is the disjoint union of the
   sets $E$, $E\circ P$, $P\circ R_E$ and $E\circ P\circ R_E$.
   Passing to (Dirichlet) generating functions, we obtain
   $$
   R_E(s) = E(s)+\big(1+E(s)\big)P(s)\big(1+R_E(s)\big).
   $$

   Similarly, $R$ is the disjoint union of the composition $(1)$, and
   the sets $K$, $R_E$, $K\circ R_E$ and $K\circ R_E\circ R$.  Hence
   $$
   R(s)=\big(1+K(s)\big)\big(1+R_E(s)\big)+K(s)R_E(s)R(s).
   $$
   Extracting $P(s)$ and observing $E(s)=K(s)=\zeta(s)-1$ we obtain
   Equation~\eqref{eq:1}.  Equation~\eqref{eq:3} can be derived by
   combining Equations~\eqref{eq:R1} and~\eqref{eq:R}.
   Equation~\eqref{eq:2} then follows from Equations~\eqref{eq:1}
   and~\eqref{eq:3}.
 \end{proof}
 \section{The number of ribbon Schur functions of given size and
   length}
 \label{sec:length}

 Apart from the size of a composition, the most natural statistic that
 comes to mind is its length.  In this section we derive an expression
 for the number of normalised compositions with given size and given
 length.

 By \ref{thm:monoid-size-length}, it is possible to determine the
 length of a composition of compositions, knowing the size and the
 length of the factors.  However, since the length of a composition of
 compositions is neither multiplicative or additive, we cannot expect
 a result as appealing as in Theorem~\ref{thm:size}.

 Let us first collect some elementary results:
 \begin{prop}
   Let $C_n(x)=\sum_{\alpha\in C, |\alpha|=n} x^{l(\alpha)}$ be the
   ordinary generating function of all compositions of size $n$, where
   $x$ marks length.  Similarly, let $S_n(x)=\sum_{\alpha\in S,
     |\alpha|=n} x^{l(\alpha)}$ the generating function of symmetric
   compositions, and $L^\times_n(x)=\sum_{\alpha\in L^\times,
     |\alpha|=n} x^{l(\alpha)}$ the generating function of asymmetric
   lexicographic minimal compositions.  Then
   \begin{align}
     C_n(x)&=x(1+x)^{n-1},\tag{\url{http://oeis.org/A007318}}\\
     S_n(x)&=
     \begin{cases}
       x(1+x)(1+x^2)^{(n-2)/2} &\text{$n$ even}\\
       x(1+x^2)^{(n-1)/2} &\text{$n$ odd}, 
     \end{cases}\tag{\url{http://oeis.org/A051159}}\\
     L^\times_n(x)&=1/2\left(C_n(x)-S_n(x)\right).
     \tag{\url{http://oeis.org/A034852}}
   \end{align}
 \end{prop}

 \begin{thm}
   Let $R_n(x)=\sum_{\rho\in R, |\rho|=n} x^{l(\rho)}$ be the ordinary
   generating function of normalised compositions of size $n$, where
   $x$ marks length.  Similarly, let $R^1_n(x)=\sum_{\rho\in R^1,
     |\rho|=n} x^{l(\rho)}$ be the ordinary generating function of
   (normalised) compositions whose first factor in the irreducible
   factorisation is asymmetric, and all remaining factors (if any) are
   symmetric.  Then we have
   \begin{align}\label{eq:R1solution}
     R^1_n(x)&=\sum_{\substack{k\geq 0\\ %
         1=d_0|d_1|\dots|d_k|n \\ %
         \text{\shrink{$d_i\neq d_{i+1}$ for
             $i\in\{0,\dots,k-1\}$}}}}%
     (-1)^k %
     L^\times_{n/d_k}(x^{d_k}) %
     \prod_{i=0}^{k-1} C_{d_{i+1}/d_i}(x^{d_i})/x^{d_i}
     \intertext{and}\label{eq:Rsolution}%
     R_n(x)&=\sum_{\substack{k\geq0 \\ %
         d_1|d_2|\dots|d_{k+1}=n \\ %
         \text{\shrink{$d_i\neq d_{i+1}$ for $i\in\{1,\dots,k\}$}}}}%
     S_{d_1}(x) %
     \prod_{i=1}^k R^1_{d_{i+1}/d_i}(x^{d_i})/x^{d_i}.
   \end{align}
 \end{thm}
 \begin{proof}
   We reuse the decompositions from the proof of
   Theorem~\ref{thm:size}.  From Equation~\eqref{eq:R1b}, we obtain
   the equality of sets (subscripts denoting the size of the
   compositions we are restricting our attention to)
   $$ 
   L^\times_n = \bigDisjointUnion_{d|n} C_d\circ R^1_{n/d}.
   $$
   Since $l(\alpha\circ\beta) = l(\alpha)-\size{\alpha} +
   \size{\alpha}l(\beta)$, it follows that
   \begin{equation}
     \label{eq:R1-length}
     L^\times_n(x) = \sum_{d|n} C_d(x)x^{-d}R^1_{n/d}(x^d).    
   \end{equation}
   Equation~\eqref{eq:R1solution} then follows from
   Equation~\eqref{eq:solveR1} in Lemma~\ref{lem:inverse} below, with
   $A_n(x)=L^\times_n(x)$, $B_n(x)=C_n(x)/x^n$ and $C_n(x)=R^1_n(x)$.

   Similarly, from Equation~\eqref{eq:R}, we obtain the equality of
   sets
   $$
   R_n=S_n\disjointUnion \bigDisjointUnion_{d|n, d\neq n}R_d\circ R^1_{n/d},
   $$
   and therefore
   \begin{equation}
     \label{eq:R-length}
     R_n(x)=S_n(x)+ \sum_{d|n, d\neq n}R_d(x)x^{-d}R^1_{n/d}(x^d).
   \end{equation}
   Equation~\eqref{eq:Rsolution} then follows from
   Equation~\eqref{eq:solveR} in Lemma~\ref{lem:inverse} below, with
   $A_n(x)=R_n(x)$, $B_n(x)=S_n(x)$ and $C_n(x)=R^1_n(x)/x$.
 \end{proof}
 \begin{rmk}
   Note that for actually computing the generating function for
   normalised compositions using a computer,
   Equations~\eqref{eq:R1-length} and \eqref{eq:R-length} may be
   easier to implement than the \lq explicit\rq\ expressions given in
   the statement of the theorem.

   Again, we can refine the count be marking the number of asymmetric
   irreducible factors with an additional variable $z$: every summand
   in Equation~\eqref{eq:Rsolution} has to be multiplied by $z^k$,
   since every composition in $R^1_n$ contains exactly one asymmetric
   irreducible factor.
 \end{rmk}

 \begin{lem}
   \label{lem:inverse}
   Suppose that $B_1(x)=1$ and
   $$
   A_n(x)=\sum_{d|n} B_d(x) C_{n/d}(x^d).
   $$
   Then we have
   \begin{equation}
     \label{eq:solveR1}
     C_n(x)=%
     \sum_{\substack{k\geq 0\\ 1=d_0|d_1|\dots|d_k|n \\ %
         \text{\shrink{$d_i\neq d_{i+1}$ for $i\in\{0,\dots,k-1\}$}}}}%
     (-1)^k %
     A_{n/d_k}(x^{d_k}) %
     \prod_{i=0}^{k-1} B_{d_{i+1}/d_i}(x^{d_i}).
   \end{equation}
   Given
   $$
   A_n(x)=B_n(x)+\sum_{d|n, d\neq n} A_d(x) C_{n/d}(x^d),
   $$
   we have
   \begin{equation}
     \label{eq:solveR}
     A_n(x)=%
     \sum_{\substack{k\geq 0 \\ %
         d_1|d_2|\dots|d_{k+1}=n \\ %
         \text{\shrink{$d_i\neq d_{i+1}$ for $i\in\{1,\dots,k\}$}}}}%
     B_{d_1}(x) %
     \prod_{i=1}^k C_{d_{i+1}/d_i}(x^{d_i}).
   \end{equation}
 \end{lem}
 \begin{proof}
   We prove the statements by induction on $n$.  For $n=1$, the
   hypothesis is $A_1(x)=B_1(x)C_1(x)=C_1(x)$, and the right hand side
   of Equation~\eqref{eq:solveR1} indeed evaluates to $A_1(x)$.

   Now suppose that Equation~\eqref{eq:solveR1} holds for $n<N$.  Then
   \begin{align*}
     C_N(x)&=A_N(x)-\sum_{1<d|N} B_d(x) C_{N/d}(x^d)\\
     &=A_N(x)-\sum_{1<d|N} B_d(x)%
     \sum_{\substack{k\geq 0\\1=d_0|d_1|\dots|d_k|N/d \\ %
         \text{\shrink{$d_i\neq d_{i+1}$ for
             $i\in\{0,\dots,k-1\}$}}}}%
     (-1)^k %
     A_{N/(d_k d)}(x^{d_k d}) %
     \prod_{i=0}^{k-1} B_{d_{i+1}/d_i}(x^{d_i d}).
     \intertext{Substituting $d'_{i+1} = d_i d$ we obtain}
     C_N(x)&=A_N(x)-\sum_{1<d|N} B_d(x)%
     \sum_{\substack{k\geq 0\\d=d'_1|d'_2|\dots|d'_{k+1}|N \\ %
         \text{\shrink{$d'_i\neq d'_{i+1}$ for
             $i\in\{1,\dots,k\}$}}}}%
     (-1)^k %
     A_{N/(d'_{k+1})}(x^{d'_{k+1}}) %
     \prod_{i=1}^k B_{d'_{i+1}/d'_i}(x^{d'_i}) \\
     &=A_N(x)-%
     \sum_{\substack{k\geq 0\\1=d'_0|d'_1|\dots|d'_{k+1}|N \\ %
         \text{\shrink{$d'_i\neq d'_{i+1}$ for
             $i\in\{0,\dots,k\}$}}}}%
     (-1)^k %
     A_{N/(d'_{k+1})}(x^{d'_{k+1}}) %
     \prod_{i=0}^{k} B_{d'_{i+1}/d'_i}(x^{d'_i}).
   \end{align*}
   The final expression is equivalent to the claimed
   Equation~\eqref{eq:solveR1}, since $A_N(x)$ is precisely the
   summand corresponding to the chain $1=d'_0|N$.

   Equation~\eqref{eq:solveR} can be shown using the same strategy,
   the calculations are actually a bit easier.
 \end{proof}

 \section*{Acknowledgements}
 I would like to express my gratitude to Stephanie van Willigenburg
 for suggesting the problem, and her enthusiasm.

 \bibliography{math}

\begin{thebibliography}{1}

\bibitem{MR2233132}
Louis~J. Billera, Hugh Thomas, and Stephanie van Willigenburg.
\newblock Decomposable compositions, symmetric quasisymmetric functions and
  equality of ribbon {S}chur functions.
\newblock {\em Advances in Mathematics}, 204(1):204--240, 2006,
  math.CO/0405434.

\bibitem{MR2459362}
Manuel Maia and Miguel M{\'e}ndez.
\newblock On the arithmetic product of combinatorial species.
\newblock {\em Discrete Math.}, 308(23):5407--5427, 2008, math.CO/0503436.

\bibitem{MR2500893}
Peter R.~W. McNamara and Stephanie van Willigenburg.
\newblock Towards a combinatorial classification of skew {S}chur functions.
\newblock {\em Transactions of the American Mathematical Society},
  361(8):4437--4470, 2009, math.CO/0608446.

\bibitem{MR2353252}
Victor Reiner, Kristin~M. Shaw, and Stephanie van Willigenburg.
\newblock Coincidences among skew {S}chur functions.
\newblock {\em Advances in Mathematics}, 216(1):118--152, 2007,
  math.CO/0602634.

\bibitem{OEIS}
Neil J.~A. Sloane.
\newblock The on-line encyclopedia of integer sequences.
\newblock {\em Notices of the American Mathematical Society}, 50(8):912--915,
  2003.
\newblock \url{http://www.research.att.com/~njas/sequences}.

\bibitem{EC2}
Richard~P. Stanley.
\newblock {\em Enumerative combinatorics. {V}ol. 2}, volume~62 of {\em
  Cambridge Studies in Advanced Mathematics}.
\newblock Cambridge University Press, Cambridge, 1999.
\newblock With a foreword by Gian-Carlo Rota and appendix 1 by Sergey Fomin.

\bibitem{MR2176598}
Stephanie van Willigenburg.
\newblock Equality of {S}chur and skew {S}chur functions.
\newblock {\em Annals of Combinatorics}, 9(3):355--362, 2005, math.CO/0410044.

\end{thebibliography}
 \bibliographystyle{hplain}

\end{document}